\documentclass[12pt]{amsart}
\usepackage{amssymb,amsmath,amsthm,bm}
\usepackage[colorinlistoftodos,prependcaption,textsize=tiny]{todonotes}
 \definecolor{darkblue}{RGB}{0,0,160}
\usepackage{fouriernc} 
\usepackage[colorlinks=true,allcolors=darkblue]{hyperref}
\DeclareSymbolFont{usualmathcal}{OMS}{cmsy}{m}{n}
\DeclareSymbolFontAlphabet{\mathcal}{usualmathcal}

\usepackage[T1]{fontenc}

\usepackage{color}

\usepackage{parskip}
\usepackage{setspace}

\usepackage{amsmath,amsthm,amscd,amssymb, mathrsfs}

\usepackage{latexsym}

\numberwithin{equation}{section}

\theoremstyle{plain}
\newtheorem{theorem}{Theorem}[section]
\newtheorem{lemma}[theorem]{Lemma}

\newtheorem{conjecture}[theorem]{Conjecture}

\theoremstyle{definition}

\theoremstyle{remark}

\newtheorem{case[theorem]}{Case}

\title[\parbox{14cm}{\centering{Group action and applications\hspace{1in}}} \quad]{Group action and $L^2$-norm estimates of geometric problems}
\author{Thang Pham}

\address{Hanoi University of Science, Vietnam National University}
\email{thangpham.math@vnu.edu.vn}


\begin{document}
\begin{abstract} 
In 2017, by using the group theoretic approach and tools from discrete Fourier analysis, Bennett, Hart, Iosevich, Pakianathan, and Rudnev obtained a number of results on the distribution of simplices and sum-product type problems. The main purpose of this paper is to give a series of new applications of their powerful framework, namely, we focus on the product and quotient of distance sets, the $L^2$-norm of the direction set, and the $L^2$-norm of scales in difference sets. 

 \end{abstract}

\maketitle
\section{Introduction}
Let $\mathbb{F}_q$ be a finite field of order $q$, where $q$ is a prime power.  Let $O(d)$ be the orthogonal group of $d$ by $d$ matrices with entries in $\mathbb{F}_q$. 

Given an integer $k\ge 1$, we say that two $k$-simplices in $\mathbb{F}_q^d$ with vertices $(x_1, \ldots, x_{k+1})$ and $(y_1, \ldots, y_{k+1})$ are in the same congruence class if $||x_i-x_j||=||y_i-y_j||$ for all $1\le i\ne j\le k+1$. Here $||x_i-x_j||=(x_{i1}-x_{j1})^2+\cdots+(x_{id}-x_{jd})^2$. This is equivalent to say that there exist $\theta\in O(d)$ and $z\in \mathbb{F}_q^d$ such that $y_i=\theta x_i+z$ for all $1\le i\le k+1$. Given a set $E\subset \mathbb{F}_q^d$, we denote the set of congruence classes of $k$-simplcies determined by $E$ by $T_k^d(E)$. 

The question of finding the smallest threshold $\alpha$ such that $|T_k^d(E)|\gg q^{\binom{k+1}{2}}$ whenever $|E|\gg q^\alpha$ has a rich history, for instance, see \cite{S1, S2, S3, S4, S5}. The best current result is due to Bennett, Hart, Iosevich, Pakianathan, and Rudnev in \cite{BII}. More precisely, by using the group theoretic approach, they proved that for all $1\le k\le d$, if $|E|\gg q^{d-\frac{d-1}{k+1}}$, then $|T_k^d(E)|\gg q^{\binom{k+1}{2}}$. In two dimensions, they are able to obtain  better exponents, namely, $8/5$ for $k=2$ and $4/3$ for $k=1$. We also note that when $k=1$, this problem is known as the Erd\H{o}s-Falconer distance problem in the literature, we refer the interested reader to \cite{MPPRS, KPV} for the recent progress. 

It is worth noting that the group action techniques are not only useful in discrete setting, but it is also very powerful in the continuous setting. One example, we have to mention here, is an $L^2$ identity due to Liu \cite{LiuL2}, which plays an important role in the recent breakthrough on the pinned Falconer distance problem \cite{DIOWZ, GIOW}. Similar results of this approach can also be found in \cite{GILP, GIM, Liu2}. 

The main purpose of this paper is to provide a series of new applications of this group theoretic approach, namely, on the product and quotient of distance sets, the $L^2$-norm of the direction set, and the $L^2$-norm of scales in difference sets. 

Throughout this paper, we denote the set of square elements in $\mathbb{F}_q$ by $(\mathbb{F}_q)^2$, and by $X\gg Y$ we mean there exists an absolutely positive constant $C$ such that $X\ge CY$. 
\subsection{Product and Quotient of distance sets}
For $E\subset \mathbb{F}_q^2$, the distance set $\Delta(E)$ is the set of all distances determined by pairs of points in $E$, namely, 
\[\Delta(E):=\left\lbrace ||x-y||\colon x, y\in E \right\rbrace.\]
We define the product, quotient, and sum of $\Delta(E)$ as follows: 
\[\frac{\Delta(E)}{\Delta(E)}:=\left\lbrace \frac{a}{b}\colon a, b\in \Delta(E), ~b\ne 0  \right\rbrace, ~\Delta(E)\cdot \Delta(E):=\left\lbrace a\cdot b\colon a, b\in \Delta(E) \right\rbrace,\] and \[\Delta(E)+\Delta(E):=\left\lbrace a+b\colon a, b\in \Delta(E) \right\rbrace.\]

It was proved in \cite{S4} that if $|E|\gg q^{4/3}$, then $|\Delta(E)|\gg q$. This result implies immediately that $|\Delta(E)+\Delta(E)|,~ |\Delta(E)\cdot \Delta(E)|,~ |\Delta(E)/\Delta(E)|\gg q$. In the prime field setting, one can improve the exponent $4/3$ to $5/4$ by using the recent result due to Murphy et al. in \cite{MPPRS}. In general, as in the original distance problem, the conjectured exponent is expected to be $1$. To support this claim, one can take $q=p^2$ and $E=\mathbb{F}_p^2$, then we have $|E|=q$ and $|\Delta(E)+\Delta(E)|,~ |\Delta(E)\cdot \Delta(E)|,~ |\Delta(E)/\Delta(E)|=p=q^{1/2}$.

For the sum of distance sets, the exponents $8/7$ and $11/10$ have been achieved over arbitrary finite fields and prime fields in \cite{P2017} and \cite{CKP}, respectively.

For the quotient of distance sets, by using discrete Fourier analysis, Iosevich, Koh, and Parshall \cite{IKP} proved that if $|E|\ge 9q$, then 
\begin{equation}\label{quotient-IKP}
\frac{\Delta(E)}{\Delta(E)}=\mathbb{F}_q.
\end{equation}
They also constructed examples to show that the condition $|E|\gg q$ is optimal. 

Using the group theoretic approach, we prove the following theorem on the size of $\Delta(E)\cdot \Delta(E)$. 
\begin{theorem}\label{thm1}
For $E\subset \mathbb{F}_q^2$ with  $|E|\gg q^{8/7}$, we have
\[|\Delta(E)\cdot \Delta(E)|\gg q.\]
\end{theorem}

One might wonder if the exponent $1$ can be attained by adapating the method in the paper \cite{IKP}. However, Iosevich and Koh \cite{IK2017} showed that the two problems do not behave in the same way, and they actually obtained the exponent $5/4$ in two dimensions.

Compared to the sum of distance sets, it is necessary to say that the study of the product set would be much harder. The main reason comes from an observation that $\Delta(E)+\Delta(E)=\Delta(E\times E)$, so the sum set can be reduced to the original distance problem for Cartesian product sets, which possesses some nicely additive structures. 

On Iosevich, Koh, and Parshall's result of $\Delta(E)/\Delta(E)$, there are two perspectives we want to mention here: their proof is very sophisticated, and it only tells us that the quotient contains the whole field, so given an element $r\in \mathbb{F}_q$, it is not clear how many tuples $(a, b, c, d)\in E^4$ such that $||a-b||/||c-d||=r$. Using the group theoretic method, we prove the following.
\begin{theorem}\label{thm-quotient}
For $E\subset \mathbb{F}_q^2$ with $q\equiv 3\mod 4$. Assume that $|E|\gg q$, then for each non-zero square $r$ in $\mathbb{F}_q$, the number of quadruples $(a, b, c, d)\in E^4$ such that $||a-b||/||c-d||=r$ is at least $\gg |E|^4q^{-1}$. In particular, 
\[\left(\mathbb{F}_q\right)^2\subseteq \frac{\Delta(E)}{\Delta(E)}.\]
\end{theorem}
We remark that the condition $q\equiv 3\mod 4$  is needed in the proof of this theorem, which helps to avoid pairs of zero distances.
\subsection{Distribution of directions}
Given $E\subset \mathbb{F}_q^2$, the $L^2$-norm of the direction set bounds the number of quadruples $(u, v, x, y)\in E^4$ such that 
\begin{equation}\label{direction-vv}(u-v)=\lambda \cdot (x-y), ~x\ne y, ~u\ne v,\end{equation}
for some $\lambda\ne 0$. We denote the number of such quadruples by $L^2(D_E)$. In the next theorem, we give an upper bound for this quantity. 
\begin{theorem}\label{thm:direction}
For $E\subset \mathbb{F}_q^2$, we have 
\[\left\vert L^2(D_E)-\frac{|E|^4}{q}\right\vert \ll q^2|E|^2.\]
\end{theorem}
If we are interested in the number of directions spanned by a set $E\subset \mathbb{F}_q^2$, then Theorem \ref{thm:direction} says that we have at least $\gg q$ directions as long as $|E|\gg q^{3/2}$. However, it is known in the literature that the condition $|E|\gg q$ would be sufficient,  see \cite{IMP} for example. When the size of $E$ is very small, say $|E|\ll q$, we refer the reader to \cite{Sz} for related results.

To the best of our knowledge, we are not aware of any known results on the magnitude of $L^2(D_E)$ for general sets $E$. If $E=A\times A$ with $A\subset \mathbb{F}_q$, then $L^2(D_E)$ is almost equal to the number of $(x_1, x_2, \ldots, x_8)\in A^8$ such that $(x_1-x_2)(x_3-x_4)=(x_5-x_6)(x_7-x_8)$. The number of such tuples has been considered in a series of papers on sum-product problems, we refer the reader to \cite{MPRRS} and references therein for more details. 

It follows from Theorem \ref{thm:direction} that $L^2(D_E)=(1+o(1))|E|^4/q$ whenever $|E|\gg q^{3/2}$. We do not believe this result is sharp, and offer the following conjecture. 
\begin{conjecture}
Let $E$ be a subset of $\mathbb{F}_q^2$. If $|E|\gg q^{1+\epsilon}$ for any $\epsilon>0$, then 
\[L^2(D_E)=(1+o(1))\frac{|E|^4}{q}.\]
\end{conjecture}

To support this conjecture, we provide two examples as follows. The first example says that the conjecture does not hold when $|E|\ll q$. More precisely, assume $q=p^2$, take $E=\mathbb{F}_p^2$, then a direct computation gives $|E|=q$ and $L^2(D_E)\sim p^7>|E|^4q^{-1}=p^6.$ The second example says that one can construct sets of arbitrary large (more than $q$) such that $L^2(D_E)\sim |E|^4/q$. To see this, let $A\subset \mathbb{F}_q$ and set $E=\mathbb{F}_q\times A$. By a direct computation, we have $L^2(D_E)\sim |A|^4q^3=|E|^4/q$. 

\subsection{Scales in difference sets}
For $E\subset \mathbb{F}_q^2$, an element $\lambda\in \mathbb{F}_q$ is called a scale in the difference set $E-E$ if there exist $u_1, v_1, u_2, v_2\in E$ such that 
\[(u_1-v_1)=\lambda(u_2-v_2).\]
We denote the set of all scales in $E-E$ by $\frac{E-E}{E-E}$. In this paper, we are interested in bounding the $L^2$-norm, i.e. the number of tuples $(u_1, v_1, u_2, v_2, u_3, v_3, u_4, v_4)\in E^8$ such that 
\[(u_1-v_1)=\lambda(u_2-v_2), ~(u_3-v_3)=\lambda(u_4-v_4),\]
for some $\lambda\ne 0$. We denote the number of such tuples by $L^2(S_E)$.

With the same approach, we have an upper bound for the $L^2$-norm of scales. 
\begin{theorem}\label{thm-scale}
For $E\subset \mathbb{F}_q^2$, we have 
\[\left\vert L^2(S_E)- \frac{|E|^8}{q^3}\right\vert \ll |E|^6+q^2|E|^5.\]
\end{theorem}
For the set of scales, it is not hard to show that if $|E|\gg q^{3/2}$, then  $E-E/E-E$ covers the whole field. Roughly speaking, under the condition $|E|\gg q^{3/2}$, one can find a line with at least $q^{1/2}$ points from $E$, then without loss of generality, we assume that line is defined by $y=0$. So the problem is reduced to one dimensional version, namely, $A-A/A-A$ for some $A\subset \mathbb{F}_q$. It is well-known that $A-A/A-A=\mathbb{F}_q$ when $|A|\gg q^{1/2}$, see \cite[Lemma 1]{balog} for example. We leave this to the reader for checking in detail. The exponent $3/2$ is also optimal, for example, take $q=p^2$ and $E=\mathbb{F}_p\times \mathbb{F}_q$, then $\left\vert\frac{E-E}{E-E}\right\vert=p=o(q)$. 

Theorem \ref{thm-scale} infers that $L^2(S_E)=(1+o(1))|E|^8/q^3$ whenever $|E|\gg q^{5/3}$. While it is hard to believe that the exponent $5/3$ is sharp, it cannot be improved to lower than $3/2$. Indeed, take $q=p^2$ and $E=\mathbb{F}_q\times \mathbb{F}_p$, then we can compute that $L^2(S_E)\sim p|E|^6=q^{9+1/2}>q^9=q^{-3}|E|^8$. Thus, it is plausible to make the following conjecture. 

\begin{conjecture}
Let $E$ be a subset of $\mathbb{F}_q^2$. If $|E|\gg q^{\frac{3}{2}+\epsilon}$ for any $\epsilon>0$, then 
\[L^2(S_E)=(1+o(1))\frac{|E|^8}{q^3}.\]
\end{conjecture}
The final remark is that all results in this paper can be extended to higher dimensions in the same way, but the proofs will become much complicated, for instance, in the proofs of Theorems \ref{thm1} and \ref{thm-quotient}, one has to count the number of pairs of zero distances, which can be done by using a number of results from Restriction theory \cite{IKPS}. Therefore, to keep this paper simple, we only present results in two dimensions.
\section{Notations from discrete Fourier analysis}
In this section, we recall some notations and results from discrete Fourier analysis which will be needed for our coming proofs in next sections. 

Given a complex function $f\colon \mathbb{F}_q^d\to \mathbb{C}$, the Fourier transform of $f$ is defined by 
\[\widehat{f}(\xi)=\frac{1}{q^d}\sum_{x\in \mathbb{F}_q^d}f(x)\chi(-x\cdot \xi),\]
where $\chi$ is a fixed non-trivial additive character of $\mathbb{F}_q$. 

Using the orthogonality property of $\chi$, i.e. 
\[\sum_{x\in \mathbb{F}_q^d}\chi(x\cdot \xi)=\begin{cases}0&\mbox{if}~\xi\ne 0,\\
q^d&\mbox{if}~\xi=0,\end{cases}\]
one can prove that 
\[f(x)=\sum_{\xi\in \mathbb{F}_q^d}\widehat{f}(\xi)\chi(\xi\cdot x).\]
In this setting, the Plancherel formula reads as 
\[\sum_{\xi\in \mathbb{F}_q^d}|\widehat{f}(\xi)|^2=q^{-d}\sum_{x\in \mathbb{F}_q^d}|f(x)|^2.\]
When $f$ is the indicator of a given set $E\subset \mathbb{F}_q^d$, one has 
\[\sum_{\xi\in \mathbb{F}_q^d}|\widehat{E}(\xi)|^2=q^{-d}|E|.\]
The following lemma is known as the finite field analog of the spherical average in the classical Falconer distance problem \cite[Chapter 3]{MAT}.
\begin{lemma}\label{spherical-average}
\[\max_{t\in \mathbb{F}_q\setminus\{0\}}\sum_{||\eta||=t}|\widehat{E}(\eta)|^2\ll \frac{|E|^{3/2}}{q^3}.\]
\end{lemma}
A proof of this lemma can be found in \cite[Lemma 4.4]{S4}.
\section{Proof of Theorem \ref{thm1}}
It is well-known that for $u, v, x, y\in \mathbb{F}_q^d$, if $||u-v||=||x-y||$, then there exist $\theta\in O(d)$ and $z\in \mathbb{F}_q^d$ such that 
\[\theta u+z=x, ~\theta v+z=y.\]
We now show that a similar statement holds for the case of product of distance sets.
\begin{lemma}
Set
\[G_1 = \bigg\{ diag(r_1\theta_1, r_2\theta_2)\in M_{2d\times 2d}\colon r_1\cdot r_2=1, \theta_1, \theta_2\in O(d)\bigg\},\]
and $G=G_1\times \mathbb{F}_q^{2d}$. For $x_1, x_2, x_3, x_4, y_1, y_2, y_3, y_4\in \mathbb{F}_q^d$, if 
\[||x_1-y_1||\cdot ||x_2-y_2||=||x_3-y_3||\cdot ||x_4-y_4||\ne 0,~ \frac{||x_4-y_4||}{||x_2-y_2||}\in (\mathbb{F}_q)^2,\]
then there exists $g\in G$ such that 
\[g(x_1, x_2)^T=(x_3, x_4)^T, ~g(y_1, y_2)^T=(y_3, y_4)^T.\]
\end{lemma}
\begin{proof}
If 
\[||x_1-y_1||\cdot ||x_2-y_2||=||x_3-y_3||\cdot ||x_4-y_4||\ne 0,\]
then 
\[\frac{||x_1-y_1||}{||x_3-y_3||}=\frac{||x_4-y_4||}{||x_2-y_2||}.\]
Set $r=\frac{||x_4-y_4||}{||x_2-y_2||}$. From our assumptions, we know that $r$ is a square, say, say, $r=t^2$ for some $t\ne 0$. This infers that
\[||(1/t)x_1-(1/t)y_1||=||x_3-y_3||.\]
Hence, we can find $\theta_1\in O(d)$ and $z_1\in \mathbb{F}_q^d$ such that 
\[(1/t)\theta_1x_1+z_1=x_3, ~(1/t)\theta_1y_1+z_1=y_3.\]
Similarly, from the fact that
\[||x_4-y_4||=||tx_2-ty_2||,\]
we can find $\theta_2\in O(d)$ and $z_2\in \mathbb{F}_q^d$ such that
\[ t\theta_2 x_2+z_2=x_4, ~t\theta_2 y_2+z_2=y_4.\]
In other words, $g=(1/t\theta_1, t\theta_2)\times (z_1, z_2)\in G$ is the element we want to find.
\end{proof}
\begin{proof}[Proof of Theorem \ref{thm1}]
Let $$A=\{(a, b)\in E\times E\colon ||a-b||\in (\mathbb{F}_q)^2\setminus \{0\}\},$$ and $$B=E\times E\setminus \left(A\cup\{(a, b)\in E\times E\colon ||a-b||=0\}\right).$$  

Let $N_0(E)$ be the number of pairs $(a, b)\in E\times E$ such that $||a-b||=0$. We know from \cite[Proposition 2.4]{KoSu} that 
\[N_0(E)\ll \frac{|E|^2}{q}+q|E|\ll |E|^2,\]
when $|E|\gg q$. Thus, either $|A|$ or $|B|$ is bounded from below by $|E|^2/2$. Without loss of generality, we assume that $|B|\gg |E|^2/2$. For any $\lambda \in (\mathbb{F}_q)^2\setminus \{0\}$, define 
\[\nu(\lambda):=\#\{(x_1, y_1, x_2, y_2)\in B^2\colon ||x_1-y_1||\cdot ||x_2-y_2||=\lambda \}.\]
We have 
\[\sum_{\lambda\in \left(\Delta(E)\cdot \Delta(E)\right)\cap (\mathbb{F}_q)^2\setminus \{0\}}\nu(\lambda)\gg |E|^4.\]
By the Cauchy-Schwarz inequality, one has 
\[\sum_{\lambda\in \Delta(E)\cap (\mathbb{F}_q)^2\setminus \{0\}}\nu(\lambda)\ll |\Delta(E)|^{1/2}\cdot \left(\sum_{\lambda\in \mathbb{F}_q}\nu(\lambda)^2\right)^{1/2}.\]
In the next step, we are going to bound $\sum_{\lambda\in \mathbb{F}_q}\nu(\lambda)^2$
from above. For $\theta=diag(r_1\theta_1, r_2\theta_2)\in G_1$ and $z=(z_1, z_2)\in \mathbb{F}_q^2\times \mathbb{F}_q^2$, we define $\mu_\theta(z)$ to be the number of tuples $(x_1, x_2, x_3, x_4)\in E^4$ such that 
\[\theta(x_1, x_2)^T+(z_1, z_2)^T=(x_3, x_4)^T. ~\]
Then we observe that 
\[\sum_{\lambda}\nu(\lambda)^2\le \sum_{\theta\in G_1, z\in \mathbb{F}_q^4}\mu_\theta(z)^2.\]
It suffices to prove that 
\[\sum_{\theta\in G_1, z\in \mathbb{F}_q^4}\mu_\theta(z)^2\ll \frac{|E|^8}{q},\]
whenever $|E|\gg q^{8/7}$.  

Set $F=E\times E$. One has 
\begin{align*}\mu_\theta(z)&=\sum_{x_1, x_2\in \mathbb{F}_q^2}F(x_1, x_2)F\left(\theta(x_1, x_2)^T+z\right)\\
&=\sum_{x_1, x_2\in \mathbb{F}_q^2}F(x_1, x_2)\sum_{m\in \mathbb{F}_q^4}\widehat{F}(m)\chi\left(m\cdot (\theta(x_1, x_2)^T+z)\right)\\
&=q^4\sum_{m\in\mathbb{F}_q^4}\widehat{F}(m)\widehat{F}(-\theta^Tm)\chi(m\cdot z).
\end{align*}
Hence, 
\[\widehat{\mu_\theta}(\xi)=q^4\widehat{F}(-\xi)\widehat{F}(\theta^T\xi).\]
Therefore, for a fixed $\theta\in G_1$,
\[\sum_{\xi\in \mathbb{F}_q^4}|\widehat{\mu_\theta}(\xi)|^2=\frac{|F|^4}{q^8}+\sum_{\xi\ne 0}|\widehat{\mu_\theta}(\xi)|^2.\]
By writing $\xi=(\xi_1, \xi_2)\in \mathbb{F}_q^2\times \mathbb{F}_q^2$ and recall that $F=E\times E$, we have 
\begin{align*}
    &\sum_{\theta\in G_1,~ \xi\in \mathbb{F}_q^4\setminus \{0\}}|\widehat{\mu_\theta}(\xi)|^2=q^8\sum_{\theta, \xi\ne 0}|\widehat{F}(-\xi)|^2|\widehat{F}(\theta^T\xi)|^2\\
    &=q^8\sum_{\theta}\sum_{\xi_1=0, \xi_2\ne 0}+q^8\sum_{\theta}\sum_{\xi\ne 0, \xi_2=0}+q^8\sum_{\theta}\sum_{\xi_1\ne 0, \xi_2\ne 0}=:I+II+III.
\end{align*}
For $I$, by Plancherel, one has 
\begin{align*}
    I&=q^8\sum_{\theta_1, \theta_2\in O(2)}~\sum_{r_1, r_2:~ r_1\cdot r_2=1}~\sum_{\xi_1=0, \xi_2\ne 0}|\widehat{E}(-\xi_1)|^2|\widehat{E}(-\xi_2)|^2|\widehat{E}(r_1\theta_1^T\xi_1)|^2|\widehat{E}(r_2\theta_2^T\xi_2)|^2\\
    &=q|E|^4\sum_{\theta_2}\sum_{r_2\ne 0}\sum_{\xi_2\ne 0}|\widehat{E}(-\xi_2)|^2|\widehat{E}(r_2\theta_2^T\xi_2)|^2\\
    &=q|E|^4\sum_{\xi_2\ne 0}|\widehat{E}(-\xi_2)|^2\sum_{r\ne 0}\sum_{||\eta||=||\xi_2||}|\widehat{E}(r\eta)|^2\\
    &\ll q^{-3}|E|^6.
\end{align*}
Similarly, $II\ll q^{-3}|E|^6.$
For $III$, we proceed as follows: 
\begin{align*}
    III&=q^8\sum_{\theta}\sum_{r_1, r_2: r_1\cdot r_2=1}\sum_{\xi_1\ne 0, \xi_2\ne 0}|\widehat{E}(-\xi_1)|^2|\widehat{E}(-\xi_2)|^2|\widehat{E}(r_1\theta_1^T\xi_1)|^2|\widehat{E}(r_2\theta_2^T\xi_2)|^2\\
    &=q^8\sum_{\theta_2}\sum_{\xi_2\ne0}\sum_{r_2\ne 0}|\widehat{E}(\xi_2)|^2|\widehat{E}(r_2\theta_2^T\xi_2)|^2\sum_{\xi_1\ne 0}|\widehat{E}(\xi_1)|^2\sum_{\eta: ||\eta||=\frac{||\xi_1||}{r_2^2}}|\widehat{E}(\eta)|^2.
\end{align*}
This implies
\[III\ll q^8\left(\sum_{\xi_1}|\widehat{E}(\xi_1)|^2\right)^3\cdot \left(\max_{t\in \mathbb{F}_q\setminus\{0\}}\sum_{||\eta||=t}|\widehat{E}(\eta)|^2\right).\]
Using Lemma \ref{spherical-average}, we obtain
\[III\ll q^{-1}|E|^{9/2}.\]
In other words, 
\[\sum_{\theta, \xi}|\widehat{\mu}_\theta(\xi)|^2\ll \frac{|E|^8}{q^5}+\frac{|E|^6}{q^3}+\frac{|E|^{9/2}}{q}.\]
This infers 
\[\sum_{\theta\in G_1, z\in \mathbb{F}_q^4}\mu_\theta(z)^2\ll \frac{|E|^8}{q}+q|E|^6+q^3|E|^{9/2}\ll \frac{|E|^8}{q},\]
whenever $|E|\gg q^{8/7}$.
\end{proof}

\section{Proof of Theorem \ref{thm-quotient}}
We note that $r=0$ is trivially in $\Delta(E)/\Delta(E).$ For $\theta\in O(2)$ and $z\in \mathbb{F}_q^2$, assume $r$ is a nonzero square, we define 
\[\eta_{\theta}(z):=\#\{(u, v)\in E\times E\colon u-\sqrt{r}\theta v=z\}.\]
We first observe that
\[\sum_{\theta\in O(2), z\in \mathbb{F}_q^2}\eta_\theta(z)=|E|^2\cdot |O(2)|.\]
So by the H\"older inequality, one has 
\[\sum_{\theta, z}\eta_{\theta}(z)^{2}\ge \frac{|E|^{4}|O(2)|}{q^{2}}.\]
On the other hand, the sum $\sum_{\theta, z}\eta_{\theta}(z)^{2}$ counts the number of tuples $(a, b, c, d)\in E^4$ such that $a-\sqrt{r}\theta c=z$ and $b-\sqrt{r}\theta d=z$, so $||a-b||=r||c-d||$. For such a tuple, we have $r\in \Delta(E)/\Delta(E)$, unless $||a-b||=||c-d||=0$.

For each tuple $(a, b, c, d)\in E^4$ such that $||a-b||=r||c-d||$, we have two possibilities.

If $a-b\ne 0$, then this tuple is counted $\sim |O(1)|\sim 1$ times in the sum $\sum_{\theta, z}\eta_{\theta}(z)^2$, since the size of the stabilizer of each non-zero element is $\sim |O(1)|$. 

If $a=b$, then $c=d$. Moreover, the number of tuples $(a, b, c, d)\in E^4$ with $a=b$ and $c=d$ in the sum $\sum_{\theta, z}\eta_{\theta}(z)^2$ is equal to $\sum_{\theta, z}\eta_{\theta}(z)=|E|^2|O(2)|$. 

From these observations, the number of tuples $(a, b, c, d)\in E^4$ such that $||a-b||=r||c-d||\ne 0$ is at least
\[\frac{1}{|O(1)|}\left(\sum_{\theta, z}\eta_{\theta}(z)^2-|E|^2|O(2)|-|O(1)|N_0\right),\]
where $N_0$ is the number of tuples with $||a-b||=r||c-d||=0$, $a\ne b$, and $c\ne d$.

Since $q\equiv 3\mod 4$, it is clear that $N_0=0$. Thus, 
\[\frac{1}{|O(1)|}\left(\sum_{\theta, z}\eta_{\theta}(z)^2-|E|^2|O(2)|-|O(1)|N_0\right)\gg \frac{|E|^4}{q},\]
whenever $|E|\gg q$.

\section{Proof of Theorem \ref{thm:direction}}
Define the group $G=\mathbb{F}_q\times \mathbb{F}_q^2$ with the operator
\[(\lambda, z_1)\ast (\beta, z_2)=(\lambda\cdot \beta, z_1+z_2).\]
For a fixed $\lambda$, we define the counting function \[\gamma_{\lambda}(z):=\#\{(a, b)\in E^2\colon \lambda a+z=b\}.\]
By a direct computation, the tuple $(u, v, x, y)$ satisfies (\ref{direction-vv}) if and only if there exist $\lambda\ne 0$ and $z\in \mathbb{F}_q^2$ such that $\lambda u+z=x$ and $\lambda v+z=y$. Thus, 
\[L^2(S_E)=\sum_{\lambda, z}\gamma_{\lambda}(z)^2.\]
For a fixed $\lambda$, we have 
\begin{align*}
    \sum_{z\in \mathbb{F}_q^2}\gamma_\lambda(z)^2&=q^2\sum_{\xi\in \mathbb{F}_q^2}|\widehat{\gamma_{\lambda}(\xi)}|^2=\frac{|\sum_{z}\gamma_\lambda(z)|^2}{q^2}+q^2\sum_{\xi\ne 0}|\widehat{\gamma_{\lambda}(\xi)}|^2\\
    &=\frac{|E|^4}{q^2}+q^2\sum_{\xi\ne 0}|\widehat{\gamma_{\lambda}(\xi)}|^2.
\end{align*}
On the other hand, 
\begin{align*}
    \gamma_\lambda(z)&=\sum_{u}E(u)E(\lambda u+z)=\sum_{u}E(u)\sum_{\xi}\widehat{E}(\xi)\chi(\xi\cdot(\lambda u+z))\\
    &=\sum_{\xi}\widehat{E}(\xi)\chi(\xi\cdot z)\sum_{u}E(u)\chi(\lambda \xi\cdot u)\\
    &=q^2\sum_{\xi}\widehat{E}(\xi)\widehat{E}(-\lambda \xi)\chi(\xi\cdot z).
\end{align*}
Hence, 
\[\widehat{\gamma_{\lambda}(\xi)}=q^2\widehat{E}(-\xi)\widehat{E}(\lambda \xi).\]
So, 
\begin{align*}
    \sum_{\xi\ne 0}|\widehat{\gamma_{\lambda}(\xi)}|^2=q^4\sum_{\xi\ne 0} |\widehat{E}(-\xi)|^2|\widehat{E}(\lambda \xi)|^2.
\end{align*}
Hence, 
\[\sum_{\lambda\ne 0} \sum_{\xi\ne 0}|\widehat{\gamma_{\lambda}(\xi)}|^2=q^4\sum_{\xi\ne 0}|\widehat{E}(\xi)|^2\sum_{m\in <\xi>\setminus \{0\}}|\widehat{E}(m)|^2,\]
where $<\xi>$ is the line generated by $\xi$.

To proceed further, we set $L:=<\xi>$ and $T_L(E):=\sum_{m\in L}|\widehat{E}(m)|^2$. 
It follows that
\begin{align*}
    T_L(E)&=q^{-4} \sum_{m\in L} \sum_{x,y\in E} \chi(m\cdot (x-y))=q^{-4} \sum_{x,y\in E} \sum_{t\in \mathbb F_q} \chi(t (\xi\cdot (x-y)))=q^{-4}(I+II),
\end{align*}
where
\[I=\sum_{x, y\in E, ~\xi\cdot (x-y)=0}\sum_{t\in \mathbb{F}_q}\chi(t (\xi\cdot (x-y)))=q\sum_{x, y\in E, ~\xi\cdot (x-y)=0} 1,\]
and 
\[II=\sum_{x, y\in E, ~\xi\cdot (x-y)\ne 0}\sum_{t\in \mathbb{F}_q}\chi(t (\xi\cdot (x-y)))=0.\]
We note that the equation $\xi\cdot (x-y)=0$ with $x, y\in E$ means that $x-y$ lies on the line with the normal vector $\xi$ and passing through the origin. Using the fact that each line in $\mathbb{F}_q^2$ contains exactly $q$ points, we have $I\le q^2|E|$. Thus, using Plancherel formula, we obtain
\[\sum_{\lambda\ne 0} \sum_{\xi\ne 0}|\widehat{\gamma_{\lambda}(\xi)}|^2\le |E|^2.\]
In other words, 
\[\left\vert \sum_{\lambda, z}\gamma_\lambda(z)^2- \frac{|E|^4}{q}\right\vert \ll q^2|E|^2.\]
This completes the proof of the theorem.
\section{Proof of Theorem \ref{thm-scale}}
Using the notations as in the previous section, we observe that 
\[L^2(S_E)=\sum_{\lambda\ne 0}\left(\sum_{z}\gamma_\lambda(z)^2\right)^2.\]

As in the previous section, we know that 
\[\sum_{z\in \mathbb{F}_q^2}\gamma_\lambda(z)^2=\frac{|E|^4}{q^2}+q^2\sum_{\xi\ne 0}|\widehat{\gamma_\lambda}(\xi)|^2,\]
and 
\[\sum_{\lambda}\sum_{\xi\ne 0}|\widehat{\gamma_\lambda}(\xi)|^2\ll |E|^2.\]
Thus, 
\[\left\vert \sum_{\lambda\ne 0}\left(\sum_{z}\gamma_\lambda(z)^2\right)^2- \frac{|E|^8}{q^3}\right\vert \ll |E|^6+q^4\sum_{\lambda}\left(\sum_{\xi\ne 0}|\widehat{\gamma_\lambda}(\xi)|^2\right)^2.\]
As we computed in the previous section, one has 
\begin{align*}
    \left(\sum_{\xi\ne 0}|\widehat{\gamma_\lambda}(\xi)|^2\right)^2=q^8\left(\sum_{\xi\ne 0}|\widehat{E}(\xi)|^2|\widehat{E}(-\lambda \xi)|^2\right)^2.
\end{align*}
Therefore, 
\begin{align*}
    &\sum_{\lambda}\left(\sum_{\xi\ne 0}|\widehat{E}(\xi)|^2|\widehat{E}(-\lambda \xi)|^2\right)^2=\sum_{\xi_1\ne 0, \xi_2\ne 0} |\widehat{E}(\xi_1)|^2|\widehat{E}(\xi_2)|^2\sum_{\lambda} |\widehat{E}(\lambda\xi_1)|^2|\widehat{E}(\lambda\xi_1)|^2\\
    &=\sum_{\xi_1\ne 0, \xi_2\ne 0}|\widehat{E}(\xi_1)|^2|\widehat{E}(\xi_2)|^2\cdot L(\xi_1, \xi_2),
\end{align*}
where 
\[L(\xi_1, \xi_2)=\sum_{\lambda}|\widehat{E}(\lambda\xi_1)|^2|\widehat{E}(\lambda\xi_2)|^2.\]
To proceed further, we need the following lemma. 
\begin{lemma}
For $\xi_1, \xi_2\ne 0$, one has 
\[L(\xi_1, \xi_2)=\sum_{\lambda}|\widehat{E}(\lambda\xi_1)|^2|\widehat{E}(\lambda\xi_2)|^2\ll \frac{|E|^3}{q^6}.\]
\end{lemma}
\begin{proof}
By the Cauchy-Schwarz inequality, we have 
\[L(\xi_1, \xi_2)\le \left(\sum_{\lambda}|\widehat{E}(\lambda \xi_1)|^4\right)^{1/2}\cdot \left(\sum_{\lambda}|\widehat{E}(\lambda \xi_2)|^4\right)^{1/2}.\]
Moreover, for $m\ne 0$,
\begin{align*}
&\sum_{\lambda}|\widehat{E}(\lambda m)|^4=\frac{1}{q^8}\sum_{x, y, z, w\in E}\sum_{\lambda}\chi(m\lambda(x+y-z-w))\\
&=\frac{1}{q^7}\sum_{x, y, z, w\in E}1_{m\cdot (x+y-z-w)=0}\\
&\ll \frac{|E|^3}{q^6},
\end{align*}
where the last inequality follows from the number of tuples $(x, y, z, w)\in E^4$ such that $x+y-z-w$ lies on the line with the normal vector $m$ and passing through the origin. This completes the proof.
\end{proof}
To continue, we apply the Plancherel formula, 
\begin{align*}
    &\sum_{\lambda}\left(\sum_{\xi\ne 0}|\widehat{E}(\xi)|^2|\widehat{E}(-\lambda \xi)|^2\right)^2=\sum_{\xi_1\ne 0, \xi_2\ne 0}|\widehat{E}(\xi_1)|^2|\widehat{E}(\xi_2)|^2\cdot L(\xi_1, \xi_2)\\
    &\ll \frac{|E|^5}{q^{10}}.
\end{align*}
In other words, we obtain
\[\left\vert \sum_{\lambda}\left(\sum_{z}\gamma_{\lambda}(z)^2\right)^2- \frac{|E|^8}{q^3}\right\vert \ll |E|^6+q^2|E|^5.\]

\section{Acknowledgement}
The author would like to thank Prof. Igor Shparlinski for a number of comments and for informing him about the paper \cite{igor}. 

The author would like to thank to Vietnam Institute for Advanced Study in Mathematics (VIASM) for the hospitality and for the excellent working condition.

\end{document}